\documentclass{article}

\usepackage{amsmath, amssymb, amsfonts, amsthm, url}

\newcommand{\comment}[1]{}

\newenvironment{equations}{\setlength\arraycolsep{1pt}\begin{eqnarray*}}{\end{eqnarray*}}

\newcounter{temp}
\newcommand{\reiterate}[2]{
  \setcounter{temp}{\value{mythm}}
  \setcounter{mythm}{#1} #2
  \setcounter{mythm}{\value{temp}}
}

\newtheorem{theorem}{Theorem}[section]
\newtheorem{lemma}[theorem]{Lemma}
\newtheorem{proposition}[theorem]{Proposition}

\newtheorem{mythm}{Theorem}

\theoremstyle{definition}
\newtheorem{definition}[theorem]{Definition}
\newtheorem{remark}[theorem]{Remark}

\newcommand{\e}{\mathbf{e}}
\newcommand{\eo}[1]{\ensuremath{\mathbf{#1}}}
\newcommand{\id}{\mathbf{d}}
\newcommand{\di}{\boldsymbol{\delta}}
\newcommand{\metric}[2]{\langle #1,#2 \rangle}
\newcommand{\gdot}{\dot{\gamma}}
\newcommand{\NormSq}[2]{\|#1\|_{#2}^{2}}
\newcommand{\Kmag}{\check{K}}

\renewcommand{\div}{\mbox{div}\,}

\renewcommand{\H}{\mathcal{H}}
\newcommand{\K}{\mathcal{K}}

\newcommand{\R}{\mathcal{R}}

\newcommand{\V}{\mathcal{V}}

\newcommand{\RR}{\mathbb{R}}

\title{On the injectivity of the X-ray transform for Anosov thermostats}
\author{Dan Jane and Gabriel P. Paternain}
\date{}

\begin{document}

\maketitle

\begin{abstract}  We consider Anosov thermostats on a closed surface and the X-ray transform on functions which are up to degree two in the velocities. We show that the subspace where the X-ray transform fails to be s-injective is finite dimensional. Furthermore, if the surface is negatively curved and the thermostat is pure Gaussian (i.e. no magnetic field is present)\comment{(i.e., no magnetic field is present)}, then the X-ray transform is s-injective. \end{abstract}

\section{Introduction}

\nocite{sharafutdinov-1994-}

\comment{Let $M$ be a \emph{surface}, a smooth, closed, orientable 2-manifold with a Riemannian metric and let $\phi_t: SM \to SM$ be a flow on its unit sphere bundle.}Let $M$ be a \emph{surface}, a smooth, closed, orientable 2-manifold with a Riemannian metric, and let $\phi_t: SM \to SM$ be a flow on its unit sphere bundle.

\comment{Given a smooth function $g:SM\to\RR$, the map $I(g)(\Gamma) = \int_\Gamma g$, where $\Gamma$ is any closed trajectory of $\phi$, is the \emph{X-ray transform} of $g$}Given a smooth function $g:SM\to\RR$, the \emph{X-ray transform} of $g$ is the map determined by $I(g)(\Gamma) = \int_\Gamma g$, where $\Gamma$ is any closed trajectory of $\phi$. Is it possible to recover $g$ if we just know its integrals along every closed orbit? Since $I$ is linear, this is obviously equivalent to asking when $I$ is injective. In general $I$ is not injective, because non-zero functions of the form $g= d/dt(u \circ \phi_t)|_{t=0}$ (called {\it coboundaries}), where
$u:SM\to\RR$ is a smooth function,
lie in the kernel of the X-ray transform.
In fact, if the flow $\phi_t$ is Anosov, the smooth Liv\v sic Theorem \cite{llave-1986-} tells us that $g$ is in the kernel of the X-ray transform
if and only if $g$ is a coboundary.

Classically, one considered recovering the function $g$ in the special case
$g \in C^\infty(M)$, a function of position only rather than both position
and velocity. In the case $g$ is linear in velocity we can write $g(x, v) =
g_i(x) v^i$, and the study of the X-ray transform has applications to
Doppler tomography, photoelasticity and other physical problems \cite{sharafutdinov-1994-}.

In this paper we are interested in studying injectivity properties of the X-ray transform for functions $g$ which are the restriction to $SM$ of a symmetric 2-tensor $q$ plus a 1-form $\sigma$. The motivation for looking at the kernel of $I$ for this class of functions comes from rigidity questions in dynamical systems and spectral geometry. Symmetric 2-tensors and 1-forms appear naturally when one linearises rigidity questions involving Riemannian metrics and magnetic fields (cf. \cite{guillemin-1980-} and \cite{dairbekov-2006-}).
There is a similar circle of questions for manifolds with boundary (motivated by the boundary rigidity problem), but we do not consider these here.

Quite a lot is known about injectivity properties of the X-ray transform for geodesic flows and in this paper we wish to extend these to more general flows on unit tangent bundles. We consider thermostats; \comment{for $i$ a rotation by $\pi/2$ in the tangent planes, let $\gamma$ be a path in $M$ determined by}so for $i$ a rotation by $\pi/2$ in the tangent planes, a generic path $\gamma$ in $M$ is determined by \begin{equation} \label{EQNnewton} \frac{D \gdot}{dt} = \lambda(\gamma, \gdot) \, i \gdot, \qquad \mbox{where } \lambda(x,v) = f(x) + \metric{\e(x)}{iv}, \end{equation} $f$ some function and $\e$ some vector field.
To simplify the exposition we shall assume that 
$\e$ has zero divergence (this is no restriction at all, cf. \cite[Lemma 2.2]{paternain-2007-}).

\comment{As above, $\phi$ will denote the induced flow, and a simple calculation confirms the sphere bundle is invariant. Notice that if $f=0$ and $\e=0$ then the system reduces to the geodesic flow; if only $\e=0$ then the system is a {\it magnetic flow}.  Let $F$ be the vector field induced by the flow.}As above, $\phi$ will denote the induced flow, and a simple calculation confirms the sphere bundle is invariant.  Let $F$ be the vector field induced by the flow. Notice that if $f=0$ and $\e=0$ then the system reduces to the geodesic flow; if only $\e=0$ then the system is a {\it magnetic flow}.

Given a pair $[q,\sigma]$, where $q$ is a symmetric 2-tensor and $\sigma$ is a 1-form, we will say that $[q,\sigma]$ is a {\it potential pair} if there exist a function $h$ on $M$ and 1-form $\psi$ such that $F(h+\psi)=q+\sigma$ (we consider restrictions of all the objects to $SM$). \comment{We note that because of the special form of $\lambda$, $F(h+\psi)$ is of degree at most two in the velocities}We note that because of the special form of $\lambda$ in (\ref{EQNnewton}), $F(h+\psi)$ is of degree at most two in the velocities. Clearly any potential pair gives rise to a function in the kernel of $I$. In Section 4 we will show that for pairs $[q,\sigma]$ there is an orthogonal decomposition into potential and {\it solenoidal} pairs. We say that $I$ is {\it s-injective} if $I[q,\sigma]=0$ implies that $[q,\sigma]$ is a potential pair. Equivalently, the restriction of $I$ to solenoidal pairs is injective.

Here is the main question we wish to address:

\medskip

\noindent{\bf Question.} If the thermostat $\phi$ is Anosov, is $I$ always s-injective?

\medskip

Note that if $I[q,\sigma]=0$, the smooth Liv\v sic Theorem \cite{llave-1986-} gives a smooth
function $u:SM\to\RR$ such that $F(u)=q+\sigma$ and the main point
we are discussing in this paper is whether $u$ is at most of {\it degree one}
in the velocities.

\comment{For geodesic flows, V. Guillemin and D. Kazhdan \cite{guillemin-1980-} proved s-injectivity when $M$ has negative curvature and the extension to higher dimensions was done  by C. Croke and V. Sharafutdinov \cite{croke-1998-}. Relaxing the constraint on the curvature, N. Dairbekov and Sharafutdinov showed that by using the Anosov structure one could modify the Pestov identity and still obtain positive results \cite{dairbekov-2003-}. Here, s-injectivity was proved up to a finite dimensional subspace.}

For geodesic flows, V. Guillemin and D. Kazhdan \cite{guillemin-1980-} proved s-injectivity when $M$ has negative curvature and the extension to higher dimensions was done  by C. Croke and V. Sharafutdinov \cite{croke-1998-}. A key ingredient in the latter proof was the \emph{Pestov Identity}. Relaxing the constraint on the curvature, N. Dairbekov and Sharafutdinov showed that by using the Anosov structure one could modify the Pestov identity and still obtain positive results \cite{dairbekov-2003-}. Here, s-injectivity was proved up to a finite dimensional subspace.

For thermostats, Dairbekov and the second author \cite{dairbekov-2007a} showed that the integrals along closed orbits determined a function $r(x) + \sigma_x(v)$, where $\sigma$ is a 1-form, up to an exact 1-form; \comment{here, exact 1-forms form the space of potential functions.}for functions of degree at most one in velocity, exact 1-forms form the space of potential functions.
This result allowed the authors to show that an Anosov thermostat on a surface
has zero entropy production if and only if the external field $\e$ has a global potential.

Let $P$ be the set of all smooth potential pairs and let $Z$ be the kernel of $I$ acting on smooth functions of the form $[q,\sigma]$.

\begin{mythm} \label{THMfindim} For an Anosov thermostat on a surface, the inclusion $P \subset Z$ has finite codimension. \end{mythm}

In other words $I$ is s-injective up to a finite dimensional subspace. It is unknown if $P=Z$ always.

A pure \emph{(Gaussian)} thermostat is one with no magnetic term. Thus $\ddot{\gamma} = \lambda(\gamma,  \dot{\gamma}) \, i \dot{\gamma}$, where $\lambda(x,v) = \metric{\e(x)}{iv}$. Pure thermostats are reversible, like the geodesic flow, whereas magnetic flows are not.

\begin{mythm} \label{THMpurethermostats} For a pure Gaussian thermostat on a negatively curved surface, the X-ray transform is s-injective. \end{mythm}

We note that thanks to a theorem M. Wojtkowski \cite{wojtkowski-2000-} a Gaussian thermostat on a surface of negative curvature is always Anosov (recall that we are assuming that $\e$ has zero divergence). However these thermostats could have very large geodesic curvatures (i.e. $\e$ could be arbitrarily large) and the usual direct approach using the Pestov identity does not work. Indeed, the bound needed is derived in \cite[Theorem 4.6]{dairbekov-2007b};
\begin{displaymath}
K - H(\lambda) + \frac{9}{5} \lambda^2 \leq 0,
\end{displaymath}
where $H$ is a vector field on $SM$ defined in Section \ref{SECTpestov}. This is clearly not true for arbitrary $\lambda$. We bypass this problem using the ideas of Sharafutdinov and G. Uhlmann in \cite{sharafutdinov-2000-}.

There is a key difficulty in extending Theorem A to higher dimensions. It is unknown if the weak bundles are transversal to the vertical fibration for thermostats in dimension $\geq 3$. This is essential in the proof of Theorem A and it is proved in dimension 2 in \cite{dairbekov-2007a}. 

For magnetic flows this property is known \cite{ppaternain-1994-} and in principle there should be no serious obstacles proving Theorem A for them using the Pestov identities developed in \cite{dairbekov-2007-}. However, even the following remains open:

\medskip

\noindent{\bf Question.} Consider an Anosov magnetic flow on a surface, so that $\e = 0$ in the equation of motion (\ref{EQNnewton}), and assume $K-H(\lambda)+\lambda^2\leq 0$. Is the X-ray transform for this flow s-injective on pairs $[q, \sigma]$? 

\section{An integral identity} \label{SECTpestov}
Let $M$ be an oriented Riemannian surface, so that the linear map $i: T_x M \to T_x M$ is well defined as follows; for $v \in T_xM$ with unit norm we require $\{v, iv\}$ to be an oriented orthonormal basis.  Recall from basic surface geometry that the sphere bundle $ \pi : SM \to M$ is a principle circle bundle; see \cite{singer-1988-}.  Let $V$ be the infinitesimal generator of the circle action.

Let $X$ be the vector field over $SM$ induced by the geodesic flow. Define a third vector field $H$ \comment{induced by}using a flow $\varphi$ orthogonal to the geodesic flow: given $(x,v) \in SM$, let $\gamma_{(x,v^\bot)}(t)$ be the unique geodesic from $x$ heading in a direction (with the orientation) perpendicular from $v$. Define $\varphi_t(x,v) = (\gamma_{(x,v^\bot)}(t),u(t))$, where $u$ is the parallel translate of $v$ along $\gamma$, and $H$ the vector field on $SM$ that $\varphi$ induces. Notice that \begin{displaymath} \metric{d\pi_{(x,v)}(H)}{iv} = 1. \end{displaymath}

A basic result in Riemannian Geometry says that $T_{(x,v)} SM$ is spanned by $\{X, H, V\}$. Again referring to \cite{singer-1988-}, we recall that \begin{equation} \label{EQNbasiccomms} [V, X] = H, \qquad [H, V] = X, \qquad [X, H] = KV, \end{equation} where $K$ is the Gaussian curvature of $M$.

As in \cite{sharafutdinov-2000-}, for real functions $\lambda$ and $c$ on $SM$ we define a modified basis $\{F, H_c, V\}$ of $T_{(x,v)}SM$ in order to explore flows other than the geodesic flow, with \begin{equation} \label{EQNadjustedbasis} F := X + \lambda V, \qquad H_c := H + cV. \end{equation}
For example, if $\lambda(x,v) = f(x)$ then the flow generated by $F$ is a magnetic flow, whereas if $\lambda(x,v) = \omega_x(v)$, for some 1-form $\omega$, then we have a Gaussian thermostat flow.  The importance of $c$ will \comment{be seen}become clear in Section \ref{SECTjacobi}.  \comment{Notice}It is worth pointing out that for a path $\gamma$ induced by the vector field $F$, \begin{displaymath} \frac{D \gdot}{dt} = \lambda \, i \gdot, \end{displaymath} as in (\ref{EQNnewton}) in the introduction.

We will need a curvature type function on the sphere bundle, $\Kmag(x,v) := K(x) - H_c(\lambda) + \lambda^2$.

\begin{remark} When $\lambda(x,v) = f(x)$, and so the flow is purely magnetic, this expression reduces to the \emph{magnetic curvature} of \cite{burns-1991-} if $c=0$. \end{remark}

\begin{lemma}[Pestov identity] For a twice differentiable function $u: SM \to \RR$, it is always true that \begin{equation} \label{EQNPestov} \begin{array}{c} \hspace{20mm} 2H_c u \cdot VFu = (Fu)^2 + (H_c u)^2 - (F(c) + c^2 + \Kmag)(V u)^2 \\ \vspace{-3mm} \\ \hspace{-20mm} \parbox[r]{37mm}{ where these terms will \newline be made to vanish later after integration. } \left\{ \hspace{5mm} \begin{array}{l} + F(H_c u \cdot V u) + V(\lambda) H_c u \cdot V u \\ - H_c(F u \cdot V u) - V(c) F u \cdot V u \\ + V(H_c u \cdot F u), \end{array} \right. \end{array} \end{equation} \end{lemma}

\begin{remark} This lemma was given in \cite{sharafutdinov-2000-} for the geodesic flow, $\lambda = 0$. Without the modified horizontal derivative consideration, so that $c=0$, a version of this lemma was given in \cite{dairbekov-2007a}. \end{remark}

\begin{proof} Combining equations (\ref{EQNbasiccomms}) and (\ref{EQNadjustedbasis}) we see that, with respect to the modified basis, the commutation relations are \begin{displaymath} [V, F] = H_c + (V(\lambda) - c) V, \end{displaymath} \begin{displaymath} [V, H_c] = -F +(V(c) + \lambda) V, \end{displaymath} \begin{displaymath} [F, H_c] = -\lambda F - c H_c + (F(c) + c^2 + \Kmag )V. \end{displaymath}

\comment{A proof can be derived from an application of these relations. There is no difficult step once we know of the existence of such an identity, and for this reason, the explicit algebra is relegated to the Appendix.}A proof can be derived from an application of these relations. There is no difficult step once we know of the existence of such an identity, so the explicit algebra has been relegated to the Appendix. \end{proof}

On $SM$ we consider the standard Liouville volume form with associated Liouville measure $\mu$.
Integrating the Pestov identity (\ref{EQNPestov}) over $SM$ with respect to $\mu$, yields the key equality. The last five terms disappear because each flow induced by the unperturbed vector fields, $X$, $H$ and $V$, leaves the volume form invariant. Thus for any function $f$ on $SM$ we have $\int X(f) d\mu = \int f \mathcal{L}_X d\mu = 0$, and similarly for $H$ and $V$.

This argument immediately implies the integral of the last term in (\ref{EQNPestov}) will be zero. \comment{Integration by parts following the above argument can be used on the other terms.}Integration by parts after applying the above argument can be used on the other terms. Using $F = X + \lambda V,$ for instance, \begin{displaymath} \int_{SM} F(H_c u \cdot Vu) d\mu = 0 + \int_{SM} \lambda V(H_c u \cdot Vu) d \mu = - \int_{SM} V(\lambda) H_c u \cdot Vu d \mu. \end{displaymath}

Explicitly then, the key integral identity derived for use in Section \ref{SECTsplitting} is  \begin{equation} \label{EQNPestovInt} 2\int_{SM} H_cu\cdot VFu \, d\mu = \NormSq{Fu}{L^2} + \NormSq{H_cu}{L^2} - \int_{SM} (F(c) + c^2 + \Kmag) (Vu)^2 \, d\mu. \end{equation}

Now we can follow the exposition in \cite{dairbekov-2003-}, where a function $c:SM \to \RR$ was chosen so as to make the contribution of the last term in the above expression arbitrarily small.  This was equivalent to using an \emph{approximate} solution of the generalised Riccati equation.  When the base manifold is a surface one can use an \emph{exact} solution. In the next section we outline how and why this can be achieved.

\section{Anosov splittings and the Riccati equation} \label{SECTjacobi}

In this section $\phi$ is the thermostat flow in the sphere bundle. We consider variations of a path $\gamma(t) = \pi \circ \phi_t (x,v)$, for some $(x,v) \in SM$. Notice $\gamma$ is defined on $M$, rather than on $SM$ as elsewhere in the paper. We let $f(s,t) = \pi(\phi_{t}(Z(s)))$ be a variation of $\gamma$; $f$ depends on a choice of curve in $TM$ called $Z$, with $\dot{Z}(0) = \xi \in T_{(x,v)}SM$.  The image of $f$ is a small rectangle covering $\gamma$. We define the \emph{Jacobi fields} (dependent on $\xi$) as $J_\xi = \partial f/ \partial s (0, \cdot)$. Then 
\begin{equation} \label{EQNjacobi} \ddot{J_\xi} + R(\dot{\gamma}, J_\xi) \dot{\gamma} - d\lambda (J_\xi, \dot{J}_\xi) \,i \dot{\gamma} - \lambda i \dot{J_\xi} = 0. \end{equation} 
The analysis is relegated to Subsection \ref{SECTthermojacobi} in the Appendix.

The importance of Jacobi fields for this paper is contained in the next Lemma.  We will need a few basic considerations of the tangent bundle, see \cite{paternain-1999-}.

On any manifold there is always a canonical \emph{vertical bundle} $\V(x,v) := \ker(d\pi_{(x,v)})$ over the tangent bundle. The choice of a \emph{horizontal bundle} $\H(x,v)$ - in order that $\H \oplus \V = TTM$ -  corresponds to a choice of connection on $TM$.

We use the Levi-Civita connection to define $\K:TTM \to TM$, the \emph{connection map}.  This will induce a choice of $\H$.  Given $\xi \in T_{(x,v)} TM$, choose a path $Z(t) = (\alpha(t), z(t)) \subset TM$ adapted to $\xi$; this means $\alpha(0) = x$, $z(0) = v$ and $\dot{Z}(0) = \xi$.  Then \begin{equation} \label{EQNconnectionmap} \K(\xi) := (\nabla_{\dot{\alpha}}z)(0). \end{equation}

\comment{We set $\H(x,v) := \ker(\K_{(x,v)})$. For the unperturbed vectors from the last section, it is an easy exercise to check that $V \in \V$ while $X, H \in \H$, and that \begin{equations} T_{(x,v)} TM &\to& T_xM \oplus T_xM \\ \xi &\mapsto& (d\pi_{(x,v)}(\xi), \K_{(x,v)}(\xi)) \end{equations} is an isomorphism. Throughout this section we use this splitting.}

We set $\H(x,v) := \ker(\K_{(x,v)})$. It is an easy exercise to check that \begin{equations} T_{(x,v)} TM &\to& T_xM \oplus T_xM \\ \xi &\mapsto& (d\pi_{(x,v)}(\xi), \K_{(x,v)}(\xi)) \end{equations} is an isomorphism; throughout this section we use this splitting. It is also immediate that for the unperturbed vectors from the last section $V \in \V$ while $X, H \in \H$.  

\begin{lemma} \label{THMflowandJ} If we choose the curve $Z \subset TM$ adapted to $\xi$, such that $\dot{Z}(0) = \xi$, and let $J_\xi$ be the Jacobi field varying with $Z$, then \begin{displaymath} d \phi_t(\xi) = (J_\xi(t), \dot{J}_\xi(t)). \end{displaymath} \end{lemma}

\begin{proof} \begin{equations} d\pi \circ d\phi_{t} (\xi) &=& d(\pi \circ \phi_t) (\xi) = \frac{d}{ds}(\pi \circ \phi_t \circ Z(s)) = J_\xi(t) \\ \K \circ d\phi_t (\xi) &=& \K \frac{d}{ds}\Big(\phi_t \circ Z(s)\Big)\Big|_{s=0} = \nabla_{d/ds(\pi \circ \phi_t \circ Z(s))} \dot{\gamma_s}|_{s=0} \\ &=& \frac{D}{ds} \frac{Df}{dt}\Big|_{s=0} = \frac{D}{dt} \frac{Df}{ds}\Big|_{s=0} = \dot{J}_\xi(t). \end{equations} \end{proof}

By definition, for an Anosov flow $\phi_t: SM \to SM$ there exists a strong stable bundle $E^- \subset T(SM)$ and a corresponding constant $0< \mu_s < 1$ such that \begin{displaymath} \| d \phi_t (\xi) \| \leq \| \xi \| e^{\mu_s t}, \qquad \forall \xi \in E^s, \, \forall t >0. \end{displaymath}  Similarly, there is a strong unstable bundle $E^+$, and $T(SM) = E^- \oplus E^0 \oplus E^+$ is a continuous splitting, where $E^0$ is generated by the flow direction. By a result of Hirsch, Shub and Pugh, \cite{hirsch-1970-}, \comment{the two weak bundles are of class $C^1$.}the two weak bundles, $E^0 \oplus E^-$ and $E^0 \oplus E^+$, are of class $C^1$. It is these weak bundles that we use in the construction.

Pick a weak stable or weak unstable bundle and denote it $E$.  We know that $E$ is invariant under the flow.  By the work of Dairbekov and the second author in \cite{dairbekov-2007a}, $E$ is always transverse to $\V$ for a thermostat flow.  Hence there exists a pointwise defined linear map $S_{(x,v)}: \H \to \V$ such that \begin{displaymath} E_{(x,v)} = \mathrm{graph}(S_{(x,v)}) = (u, S_{(x,v)}(u)) \subset \H \oplus \V \cong T_xM \oplus T_xM. \end{displaymath} Using the identification above, we view $S_{(x,v)}$ as a map $T_xM \to T_xM$. 

Notice that as $E$ is a weak bundle, certainly it contains the vector field induced by the flow; $E \ni F = X + \lambda V$. Thus $S(d\pi(X)) = \K(\lambda V)$, i.e. $S_{(x,v)}(v) = \lambda iv$. \comment{Define a new function on the sphere bundle complementing $\lambda$,}We define a new function on the sphere bundle which together with $\lambda$ determines the weak bundle, \begin{displaymath} c(x,v):= \metric{S_{(x,v)}(iv)}{iv}. \end{displaymath}

\begin{remark}  Of course, this actually defines two functions; one for the weak stable bundle, and one for the weak unstable bundle.  There should be no confusion as the arguments follow identically, and the brevity helps give a clear account. \end{remark}

If we take $\xi$ in a weak bundle $E$, then $d\phi_t(\xi) \in E \,\, \forall t$, as $E$ is flow invariant.  Since $E$ can be represented as the graph of the bundle map $S$ and $d \phi_t(\xi) = (J_\xi(t), \dot{J}_\xi(t)) \subset E$ by Lemma \ref{THMflowandJ}, we have \begin{displaymath} SJ_\xi = \dot{J}_\xi, \, \Rightarrow \, \ddot{J}_\xi = \dot{S} J_\xi + S^2 J_\xi. \end{displaymath}

\comment{We now proceed as usual, see for instance \cite{paternain-1996-}, and from (\ref{EQNjacobi}) obtain (see Appendix) \begin{equation} \label{EQNriccati} F(c) + \lambda^2 + c^2 + K - H_c(\lambda) = F(c) + c^2 + \Kmag = 0. \end{equation}}

We now proceed as usual \cite{paternain-1996-} and from (\ref{EQNjacobi}) obtain \begin{equation} \label{EQNriccati} F(c) + \lambda^2 + c^2 + K - H_c(\lambda) = F(c) + c^2 + \Kmag = 0. \end{equation} The explicit calculation is done in the Appendix.

\begin{remark} Since the weak bundles are $C^1$, \cite{hirsch-1970-}, we do not need to approximate $c$ in the derivation of the Pestov identity, as was required in \cite{dairbekov-2003-}.  This means that the Pestov identity on surfaces greatly simplifies when modified to any Anosov flow; see the next Lemma.  In general, this will not hold in higher dimensions. \end{remark}

\begin{lemma} \label{THMpestov} Let $u:SM \to \RR$ be any function.  From the discussion in this Section we can pick $c:SM \to \RR$, corresponding to either the weak unstable bundle or the weak stable bundle, such that \begin{displaymath} 2\int_{SM} H_cu\cdot VFu \, d\mu = \NormSq{Fu}{L^2} + \NormSq{H_cu}{L^2}. \qedhere \end{displaymath} \end{lemma}

\section{The energy estimates method} \label{SECTsplitting}

Throughout, $u$ and $p$ are functions on the sphere bundle.  In \cite{croke-1998-}, Croke and Sharafutdinov studied the geodesic cohomological equation, \begin{displaymath} Xu = p, \end{displaymath}on a compact negatively curved manifold. For $p$ a symmetric tensor function of the form $p(x,v) = p_{i_1 \ldots i_m}(x) v^{i_1} \cdots v^{i_m},$ it was shown that $p$ must be \emph{potential}.
In that paper, a potential function was of the form $p = \id v$, where $\id = d^s = \sigma \circ \nabla$ was covariant differentiation of a tensor followed by symmetrisation of the indices. We now outline why we cannot expect such a simple constraint for solutions of the thermostat cohomological equation \begin{displaymath} Fu = p, \end{displaymath} with $F = X + \lambda V,$ as in (\ref{EQNnewton}).
We can uniquely split a function $p$ on $SM$ into its \emph{even} and \emph{odd} parts; \begin{equations} 2p^e(x,v) &=& p(x,v) + p(x,-v) \\ 2p^o(x,v) &=& p(x,v) - p(x,-v), \end{equations}such that $p = p^e + p^o$. For the general thermostats in this section, $\lambda$ has both even and odd components. By considering the even and odd parts of the thermostat cohomological equation we immediately realise the extra $\lambda V$ term requires a more involved definition of the operator $\id$. Similarly, we should not restrict our attention just to symmetric tensor fields of some fixed degree.

We are interested in the X-ray transform $I$ 
applied to functions $p$ of the form $p(x,v) = q_x(v,v) + \sigma_{x}(v)$, where $q$ is a symmetric 2-tensor, $\sigma$ is a 1-form on $M$. If $D$ is a notation for a function space ($C^k$, $L^p$, $H^k$, etc.), then we will denote by $\mathbf D(M)$ the corresponding space of pairs $\mathbf f=[q,\sigma]$, with $q$ a symmetric covariant $2$-tensor and $\sigma$ a $1$-form, and denote by $\mathcal D(M)$ the corresponding space of pairs $\mathbf{w} = [\psi,h]$, with $\psi$ a $1$-form and $h$ a function on $M$. In particular, $\mathbf{L}^2(M)$ is the space of square integrable pairs $\mathbf f=[q,\sigma]$, and we endow this space with the norm \begin{equation}\label{norm} \|\mathbf f\|^2 =\int_M\left\{|q|^2+|\sigma|^2\right\}\,d\mbox{\rm vol}, \end{equation} with the corresponding inner product. In the space $\mathcal L^2(M)$ we will consider the norm \begin{equation}  \label{S7} \|\mathbf{w}\|^2 = \int_M \left(|\psi|^2 + h^2 \right)\,d\mbox{\rm vol}. \end{equation}

Clearly, the norm of a pair $[q,\sigma]$ in $\mathbf{L}^2(M)$ is equivalent to the norm of the corresponding function $p(x,v)$ in $L^2(SM)$. We wish to construct an operator $\id$ such that $I(\id [\psi,h])=0$ for all $[\psi,h]\in \mathcal{H}^1(M)$. For this it is enough to have $\id [\psi,h]=[q,\sigma]$, where $q+\sigma= F(h+\psi)$.



We use the hypothesis $\lambda(x,v) = f(x) + \metric{\e}{iv}$; earlier in the paper $\lambda$ could have been any once differentiable function on $SM$.
Set $\theta_{x}(v):=\metric{\e(x)}{v}$, so $\lambda=f+V(\theta)$.

We compute
$(X + \lambda V)(\psi + h) = d^s \psi + dh + f \cdot V(\psi) + V(\theta) \cdot V(\psi) + 0$.
Projecting onto the odd and even parts, this is rewritten as \begin{displaymath} \id \left( \begin{matrix} \psi \\ h \end{matrix} \right) = \left( \begin{matrix} d^s \psi + V(\theta) \cdot V \psi \\ dh + f \cdot V \psi \end{matrix} \right) = \left( \begin{matrix} d^s + V(\theta) \cdot V  & 0 \\ f \cdot V & d \cdot V \end{matrix} \right) \left( \begin{matrix} \psi \\ h \end{matrix} \right). \end{displaymath}
Hence $\id$ is the operator
\[\id = \left( \begin{matrix} d^s + V(\theta) \cdot V & 0\\ f \cdot V & d\cdot V \end{matrix} \right).\]

\begin{definition} We call a pair $[q,\sigma]\in \mathbf{L}^2(M)$ {\em potential}
if the equations $\id [\psi,h]=[q,\sigma]$ hold with $[\psi,h]\in \mathcal H^1(M)$. A pair $[q,\sigma]$ is called {\em solenoidal} if it is
$L^2$-orthogonal to all potential pairs.
\end{definition}

We let $\{\, \cdot \, \}^s$ denote symmetrisation. If a pair $[q,\sigma]$ is orthogonal to all potential pairs, then for all $\psi$ and $h$ we have:

\[\int_{M}(q,d^{s}\psi+\{V(\theta)\cdot V(\psi))\}^s+(\sigma,f\cdot V(\psi)+dh) \,d\mbox{\rm vol}=0.\]
Then
\[-\int_{M}(\delta q+V(\iota_{i\e}q)+f\cdot V(\sigma),\psi)+(\delta\sigma,h)\,d\mbox{\rm vol}=0,\]
where $\delta$ is divergence and $\iota_{i\e}q$ is the 1-form given by
$(\iota_{i\e}q)_{x}(v)=q_{x}(i\e(x),v)$.  Therefore,

\begin{equations} \delta q + V(\iota _{i\e} q) + f \cdot V(\sigma) &=& 0, \\ \delta \sigma &=&0.\end{equations}
Thus $[q,\sigma]$ is solenoidal if $\di [q,\sigma] = 0$, where \begin{displaymath} \di = \left( \begin{matrix} \delta+V \circ \iota_{i\e} & f\cdot V \\ 0  & \delta \end{matrix} \right) \end{displaymath} and $\id=-\di^{*}$.
The operator $-\di\id$ acting on pairs $\mathbf{w}=[\psi,h]$ is elliptic (the proof is the same as in
\cite[Section 3.3]{dairbekov-2006-} for the pure magnetic case)
and if the thermostat $\phi$ is transitive, its kernel is given by the pairs $[0,h]$ where $h$ is a constant. Indeed the kernel is given by those smooth
$\mathbf{w}$ 
such that 
$\id\mathbf{w}=0$, which means that $h+\psi$ is a first
integral of the thermostat. Assuming transitivity we must have $\psi=0$ and
$h$ constant.

Suppose now that $\mathbf{f}=[q,\sigma]$ is given. By solving the
equation $\di\id \mathbf{w}=\di \mathbf{f}$ for $\mathbf{w}$ and setting
$\mathbf{f}^s:=\mathbf{f}-\id \mathbf{w}$ we see that we can decompose
$\mathbf{f}$ as
\begin{equation} \label{EQNpssplitting}\eo{f} = \id \eo{w} + \eo{f}^s, \qquad \di \eo{f}^s =0 \end{equation}
where $\mathbf{f}^s$ is uniquely determined and $\mathbf{w}$ is uniquely
determined up to pairs of the form $[0,h]$ with $h$ constant.
If $\mathbf{f}$ is smooth, then $\mathbf{f}^s$ and $\mathbf{w}$
 are also smooth.
Finally we note that an Anosov thermostat on a surface is always
transitive by a result of E. Ghys \cite{ghys-1984-}.

To simplify notation, we let $|u|$ denote the $L^2$ norm of $u$, and $\|u\|$
its $H^1$ norm.

\begin{proposition} \label{THMkeyinequality} If $u \in C^{\infty}(SM) $ satisfies the cohomological equation \begin{displaymath} p = Fu, \end{displaymath} with $p(x,v) = q_x(v,v) + \sigma_{x}(v)$, then there exists a constant $C$ such that \begin{displaymath} \|u\|^2 \leq C \Big( |u|^2 + |\di [q,\sigma]| \cdot |u|\Big). \end{displaymath} \end{proposition}

\begin{proof} Substituting into the Pestov integral identity of Lemma \ref{THMpestov} (with $c$ a solution of the Riccati equation as found in Section \ref{SECTjacobi}), we see \begin{displaymath} 2\int_{SM} H_cu \cdot Vp \, d\mu = |p|^2 + |H_c u|^2. \end{displaymath}
The initial steps are algebraic manipulation. All integrals are over $SM$ with respect to the Liouville measure $\mu$.
 \begin{equations} |p|^2  + |H_cu|^2 &=& 2\int H_c u \cdot Vp
\\ &=& 2\int H_c(u \cdot Vp) - 2\int u \cdot H_cV(p)
\\ &=& 2\int cV(u \cdot Vp) - 2\int u \cdot (HVp+ cV^2p)
\\ &=& -2\int u \cdot (Vc \cdot Vp + HVp + cV^2p),
\end{equations} \begin{equation} \label{EQNhcu} \Rightarrow |p|^2 + |H_cu|^2 \leq - 2\int u \cdot \big( \, HVp + Vc\cdot Vp + cV^2p \, \big). \end{equation}
We look to bound each term on the right hand side. Note that since $SM$ is compact, the continuous function $c: SM \to \RR$ is bounded.  Since it is of class $C^1$, $Vc$ is also a bounded function.

We have a prescribed form for $p$, which gives \begin{equation} \label{EQNvp} (Vp)(x,v) = V(q_x(v,v) + \sigma_x(v)) = 2q_x(iv,v) + \sigma_x(iv). \end{equation}  So $|Vp| \leq 2 |p|$, and similarly $|V^2 p| \leq 4|p|$.

Let us bound the integral of $uHVp$.  Using (\ref{EQNvp}) there are two integrals to consider, that of $uHV\sigma$ and that of $uHVq$.  We will need the following easy results (cf. \cite[Lemmata 5.2 and 6.1]{paternain-2007-}); \begin{equation} \label{EQNdivs} X(\sigma) + HV(\sigma) = \delta \sigma, \qquad X(q) + HV(q)/2 = \delta q. \end{equation}

\begin{equations} \int u \cdot HV(\sigma) &=& \int u \cdot \delta \sigma - \int u \cdot X (\sigma)
\\ &=& \int u  \cdot \delta \sigma + \int X(u) \cdot \sigma
\\ &=& \int u \cdot \delta \sigma +\int p \cdot \sigma - \int \lambda Vu \cdot \sigma
\\ &=& \int u \cdot (\delta \sigma + V(\lambda) \sigma + \lambda V(\sigma))+ \int p \cdot \sigma
\\ &=& \int \! u \cdot (\delta \sigma + 2f\cdot V(\sigma)) + \! \int \! u \cdot V(V(\theta) \cdot \sigma - f \cdot \sigma) + \! \int \! p \cdot \sigma. \end{equations} \begin{equations} \int u \cdot HV(q)/2 &=& \int u \cdot \delta q - \int u \cdot X(q)
\\ &=& \int u \cdot \delta q + \int p \cdot q - \int \lambda Vu \cdot q
\\ &=& \! \int \! u \cdot (\delta q + V(\iota_{i\e} q)) + \! \int \! u \cdot (f V(q) - \iota_\e q) + \! \int \! p \cdot q.
\end{equations}

Notice that the first integral in each of the above calculations - switching back to the original coefficients in (\ref{EQNhcu}) - will sum to yield \begin{equation} \label{EQNhalffactor} -2\int \! u \cdot (\delta \sigma + 2 f \cdot V(\sigma)) - 4 \int \! u \cdot (\delta q + V(\iota_{i\e} q)) \,\leq 4|u| \cdot | \di [q,\sigma]|. \end{equation}  The last term in each of the two expressions are also bounded, \begin{displaymath} -2\int p \cdot \sigma - 4 \int p \cdot q = -4 \int \sigma^2/2 + q^2 \leq 0. \end{displaymath} Finally, the middle integral of both expressions is bounded by $C'|u|\cdot |p|$, by the same arguments as when dealing with the operator $V$ in (\ref{EQNvp}).

Feeding this into (\ref{EQNhcu}) we have a bound on the left hand side: \begin{equations} |p|^2 + |H_cu|^2 &\leq& - 2\int u \cdot (Vc \cdot Vp - cV^2p) \,  - 2\int u \cdot HVp
\\ &\leq& C |u| \cdot |p| + 4 |u| \cdot |\di [q,\sigma]| + 0 + C' |u| \cdot |p|.\end{equations} Here $C$ and $C'$ are constants, dependent only on $\lambda$ and $c$.  Hence the constants are determined by the choice of Anosov flow, rather than $u$ and $p$.  Let $C$ denote some generic constant greater than 1.

\begin{displaymath}|p|^2 + |H_cu|^2 \leq  C \Big( |u| \cdot |p| +  |u| \cdot |\di [q,\sigma]| \Big). \end{displaymath}

Let $c^u$ and $c^s$ be the two Riccati solutions from Section \ref{SECTjacobi}, corresponding to the weak unstable and stable bundles of the Anosov flow. Adding their corresponding inequalities when substituted into the last inequality, we have \begin{equation}  \label{EQNhcubound} |u|^2 + |p|^2 + |H_{c^u}u|^2 + |H_{c^s}u|^2 \leq C \Big( |u|^2 + |u| \cdot |p| +  |u| \cdot |\di [q,\sigma]| \Big). \end{equation}
We claim that the first Sobolev norm of $u$ is smaller than the left hand side. Recall that $\|u\|^2 := \|u\|_{H^1}^2 = |u|^2 + |Xu|^2 + |Hu|^2 + |Vu|^2$. Using the definition of $F$ and the cohomological equation, \begin{displaymath} |Xu|^2 \leq |p|^2 + |\lambda Vu|^2 \leq C(|p|^2 + |u|^2). \end{displaymath} Since the flow is Anosov, $c^s \neq c^u$ everywhere, which means the linear system \begin{displaymath} \left( \begin{matrix} H_{c^u} \\ H_{c^s} \end{matrix} \right) = \left( \begin{matrix} 1 & c^u \\ 1 & c^s \end{matrix} \right) \left( \begin{matrix} H \\ V \end{matrix} \right) \end{displaymath} has a unique pointwise solution for $H$ and $V$. Say $H = \alpha H_{c^u} + \beta H_{c^s}$. Then \begin{equations} |Hu|^2 \leq \alpha^2 |H_{c^u}u|^2 + \beta^2 |H_{c^s}u|^2
\leq C\left( |H_{c^u}u|^2 + |H_{c^s}u|^2\right).\end{equations}The bound for $|Vu|^2$ follows in the same way.

Therefore, using (\ref{EQNhcubound}), \begin{equation} \label{EQNalmostdone} \|u\|^2 \leq C \Big( |u| ^2 + |u| \cdot |p| +  |u| \cdot |\di [q,\sigma]| \Big). \end{equation}

Finally, we use the cohomological equation to bound $|p| = |Fu| \leq C \|u\|$ which turns (\ref{EQNalmostdone}) into a quadratic inequality in $\|u\|$.  Solving gives the required bound. \end{proof}

We now have the same bound on the first Sobolev space as in \cite{dairbekov-2003-}, but for Anosov thermostat flows on a surface.  Thus we have the same consequence as their paper.  We quickly outline the proof they gave.

\reiterate{0}{\begin{mythm} For an Anosov thermostat on a surface,
the inclusion $P \subset Z$ has finite codimension. \end{mythm}}

\begin{proof}  If the inclusion is not finite dimensional, take a linearly independent sequence of functions $\{z_k\}_{k=1}^\infty \in Z \backslash P$.  Using the potential-solenoidal decomposition (\ref{EQNpssplitting}), for each $k$ there exists functions such that \begin{displaymath} z_k = \id x_k + y_k, \quad \di y_k = 0. \end{displaymath}
But $z_k \in Z \Rightarrow y_k \in Z$, as $\id x_k$ is potential and therefore integrates to zero along all trajectories.  Using the Liv\v sic Theorem \cite{llave-1986-}, we can find another sequence of functions $\{u_k\}_{k=1}^\infty$ with \begin{displaymath} F u_k = y_k. \end{displaymath}  Using linearity and Gram-Schmidt, we can assume $\{u_k\}_{k=1}^\infty$ is an orthonormal set.  The $y_k$ may also need to be adjusted, but by linearity they will still be solenoidal.

Applying the key inequality, Proposition \ref{THMkeyinequality}, on $F u_k = y_k$ the $\di y_k$ term disappears:  \begin{displaymath} \|u_k\| \leq C |u_k| = C, \end{displaymath} as the $u_k$ were normalised.  Therefore in $H^1$ the set $\{u_k\}$ is bounded.  This leads to the required contradiction, as the embedding $H^1 \hookrightarrow L^2$ is compact, which would imply the infinite orthonormal sequence $u_k$ had a convergent subsequence. \end{proof}

\section{Pure thermostats in negative curvature} \label{SECTpurethermostats}

For geodesic flows on surfaces, Sharafutdinov and Uhlmann \cite{sharafutdinov-2000-} found a useful new form of the integrated Pestov identity (\ref{EQNPestovInt}). The main advantage of their modification was that they had effectively `completed the square', controlling the $\metric{Hu}{VXu}_{L^2}$ term. Thus the full power of the identity could once again be applied, as was done in \cite{croke-1998-}.

Here we exhibit a subset of thermostat flows for which the results of \cite{sharafutdinov-2000-} still hold, so that the inclusion $P \subset Z$ is an equality, $P = Z$. We proceed on a negatively curved surface with a pure thermostat flow, such that $\lambda =V(\theta)$. The equation of motion takes the form \begin{displaymath} \frac{D\dot{\gamma}}{dt} = \metric{\dot{\gamma}}{i\e}\,i\gdot \end{displaymath}
and we will assume that $\e$ is divergence free.

\begin{remark} \label{RMKevenforpure} Notice that the odd and even parts of the cohomological equation are uncoupled, unlike in Section \ref{SECTsplitting}; \begin{displaymath} F(u^e) = p^o, \quad F(u^o) = p^e. \end{displaymath} Again, we are considering coboundaries which are up to quadratic in the velocities. However, thermostatic coboundaries of up to linear order were understood in \cite[Theorem B]{dairbekov-2007a}. This leads us to consider coboundaries which are purely even, $p=q$, for some symmetric 2-tensor $q$. \end{remark}

\begin{lemma} \label{THMUhlmannsPestov} Suppose $F(u)=p$, where $u\in C^{\infty}(SM)$, $p$ is a symmetric 2-tensor and $F = X + V(\theta) \cdot V$. Let $\varphi = V^2(u) + u$. Then the integral Pestov identity (\ref{EQNPestovInt}) can be rewritten \begin{equation} \label{EQNUhlmannPestov} |F\varphi|^2 + |H_c \varphi|^2 + |2 \cdot H_c\varphi - (\theta + c) \cdot V\varphi|^2 = \int_{SM} \R_c \cdot (V\varphi)^2 \, d\mu \end{equation} where, as before, $|\cdot|$ denotes the $L^2$ norm and we have a Riccati like function \begin{displaymath} \R_c = F(\theta + c) + (\theta + 2c)(\theta + c) + K. \end{displaymath}  \end{lemma}

\begin{proof} \begin{displaymath} VFu = FVu + [V, F]u = FVu + H_c u - (\theta + c)Vu, \end{displaymath} and similarly \begin{equations} V^2Fu &=& FV^2u + [V, F]Vu + H_c V u + [V, H_c]u \\ && \qquad -(V(\theta) + V(c))Vu - (\theta + c) V^2 u
\\ &=& FV^2u + H_c Vu - (\theta + c)V^2u + H_c V u - Fu + (V(c) + V(\theta))Vu \\ && \qquad -(V(\theta) + V(c))Vu - (\theta + c) V^2 u \\ V^2p&=& FV^2u + 2H_c Vu - 2(\theta + c)V^2u - p, \end{equations} where we have used $Fu = p$ for the last line.

Denote $\varphi = V^2u + u$. Using the equation for $V^2p$, \begin{equations} F\varphi &=& FV^2u + Fu \\ &=& V^2p - 2H_cVu +2(\theta + c)V^2u + p + p \\ \\ \Rightarrow VF\varphi &=& -2VH_cVu + 2(V(\theta) + V(c)) V^2u + 2(\theta + c)V^3u + V^3p + 2Vp \\ &=& -2H_cV^2u + 2FVu + 2(\theta + c)V^3u + V^3p + 2Vp \\ &=& -2H_cV^2u + 2VFu - 2H_cu + 2(\theta + c)Vu \\ && \qquad \qquad + 2(\theta + c)V^3u + V^3p + 2Vp \\ &=& -2H_c \varphi +2(\theta + c)V\varphi + V^3p + 4Vp.\end{equations}

As mentioned in Remark \ref{RMKevenforpure}, we have assumed a prescribed form for $p \,$; that it is even and up to quadratic in the velocities. But then $V^3p+4Vp = 0$, and so \begin{equation}
 VF \varphi = -2H_c \varphi + 2(\theta + c) V \varphi. 
\label{ex}
\end{equation} 
But then \begin{equations} 2 \int_{SM} H_c \varphi \cdot VF \varphi \, d\mu &=& -4 \, \NormSq{H_c \varphi}{L^2} + 4 \int_{SM} H_c \varphi \cdot (\theta + c)V \varphi \, d \mu \\ &=& -\NormSq{2H_c \varphi - (\theta + c) V \varphi}{L^2} + \int_{SM} (\theta + c)^2 (V \varphi)^2 \, d\mu \end{equations}

Substituting this into the Pestov identity, (\ref{EQNPestovInt}) in Section \ref{SECTpestov}, transforms the left hand side when applied to the function $\varphi = V^2u + u$: \begin{displaymath} \hspace{-2cm} -\NormSq{2H_c \varphi - (\theta + c) V \varphi}{L^2} + \int_{SM} (\theta + c)^2 (V \varphi)^2 \, d\mu \end{displaymath} \vspace{-4mm} \begin{displaymath} \hspace{3cm} = \NormSq{F \varphi}{L^2} + \NormSq{H_c \varphi}{L^2} - \int_{SM} (F(c) + c^2 + \Kmag)(V \varphi)^2 \, d\mu, \end{displaymath} and all that is left is to tidy the two integrands into one Riccati type function. \begin{equations} F(c) + c^2 + \Kmag + (\theta + c)^2 &=& F(c)+ c^2 + K - H_cV(\theta) + (V \theta)^2 + (\theta + c)^2 \\ &=& F(\theta + c) + (\theta + 2c)(\theta + c) + K, \end{equations} as $\e$ is divergence free, so we can use $X\theta + HV \theta = \div \e = 0$. \end{proof}

\begin{lemma} The two functions induced by the stable and unstable weak bundles of the flow, $c^s$ and $c^u$ of Section \ref{SECTjacobi}, are such that everywhere\begin{displaymath} (\theta + c^s)(\theta + c^u)<0 \end{displaymath} on a negatively curved surface. \end{lemma}

\begin{proof} The fact that pure Gaussian thermostats are reversible gives
$c^{s}(x,-v)=-c^{u}(x,v)$. Thus
\[\theta(x,-v)+c^{s}(x,-v)=-(\theta(x,v)+c^{u}(x,v))\]
and it suffices to show that the functions $\theta+c^{s,u}$ do not vanish.
We know that
\[F(\theta+c^{s,u})+c^{s,u}(c^{s,u}+\theta)+K=0.\]
Suppose $\theta+c^s$ vanishes at $(x,v)$. Since $K<0$, then
$F(\theta+c^{s})>0$ at $(x,v)$. It follows that $(\theta+c^s)(\phi_{t}(x,v))$
vanishes for at most one value of $t\in\mathbb R$.
Recall that by a result of Ghys \cite{ghys-1984-} the closed orbits of $\phi$
are dense, hence we can find $(y,w)$ close to $(x,v)$ such that
$\theta+c^s$ vanishes at $(y,w)$ and the orbit of $(y,w)$ is closed.
If $T$ is the period of the closed orbit we have
$(\theta+c^s)(y,w)=(\theta+c^s)(\phi_{T}(y,w))=0$ which is a contradiction.

 \end{proof}

\begin{proposition} \label{THMnegativefunction} For a thermostat on a negatively curved surface, there exists a $C^1$ function $c: SM \to \RR$ such that $\R_c$ is everywhere strictly less than zero. \end{proposition}

\begin{proof} We already have two functions, $c^s$ and $c^u$ induced by the stable and unstable weak bundles of the flow, which satisfy a thermostat type Riccati equation (\ref{EQNriccati}). After a little rearrangement (and since $\e$ is solenoidal), and then adding the two almost identical equations we get \begin{equation} \label{EQNricadded} F(c^s + c^u + 2\theta) + c^s(c^s + \theta) + c^u(c^u + \theta) + 2K = 0. \end{equation}

Recall $\R_c = F(\theta + c) + (\theta + 2c)(\theta + c) + K.$ Substituting $2c = c^s + c^u$ into (\ref{EQNricadded}), \begin{equations} 0 &=& F(2c + 2\theta) + c^s(2c - c^u + \theta) + c^u(2c - c^s + \theta) + 2K \\ &=& F(2c + 2\theta) + (c^s + c^u)(2c + \theta) - 2 c^s c^u + 2K \\ &=& F(2c + 2\theta) + (2c + 2\theta)(2c + \theta) - 2\theta (2c + \theta) - 2 c^s c^u + 2K \\ &=& 2\R_c -2(\theta + c^s)(\theta + c^u). \end{equations} Hence $\R_{(c^s+c^u)/2} = (\theta + c^s)(\theta + c^u)$, which is less than zero by the previous Lemma. \end{proof}

\reiterate{1}{\begin{mythm} For a pure Gaussian thermostat on a negatively curved surface, the X-ray transform is s-injective. \end{mythm}}

\begin{proof} Let $Fu = p = q + \sigma$. Since the flow is Anosov, the simpler cohomological equation \begin{displaymath} F(u) = \sigma \end{displaymath} has a solution iff $\sigma = F(h)$ for some $h \in C^\infty(M)$, using \cite[Theorem B]{dairbekov-2007a}. So we replace $u$ by $u-h$, and assume $p = q$.

Using Proposition \ref{THMnegativefunction} in (\ref{EQNUhlmannPestov}) of Lemma \ref{THMUhlmannsPestov}, with $\varphi = V^2u + u$, \begin{displaymath} |F\varphi|^2 + |H_c \varphi|^2 + |2 \cdot H_c\varphi - (\theta + c) \cdot V\varphi|^2 \leq 0. \end{displaymath} In particular, $F\varphi$ is zero and   so $\varphi$ is an integral of the flow. But the flow is Anosov which implies the only integrals are constants.

We have $V^2u + u = k$, some constant, which can be rewritten \begin{displaymath} V^2(u - k) + (u - k) =0. \end{displaymath}
For $(x, v) \in SM$ we can write $v = (\cos \omega, \sin \omega) \in T_x M$ for some real $\omega$. Then there exists $\alpha, \beta \in C^\infty(M)$ who induce $\psi = \psi(\alpha, \beta) \in \Omega^1(M)$ such that \begin{displaymath} u(x, v) = k + \alpha(x) \cos \omega + \beta(x) \sin \omega = k + \psi(v). \end{displaymath} Thus $u = k + \psi$, a constant plus a one form. Hence for the general coboundary $p = q + \sigma$, \begin{displaymath} Fu = p \, \Rightarrow \, u(x,v) = h(x) + \psi_x(v). \end{displaymath} \end{proof}

\begin{remark} The hypotheses in Theorem B can be somewhat relaxed.
All that is needed is that
the thermostat $\phi$ is Anosov and $(\theta + c^s)(\theta + c^u)\leq 0$ everywhere.
\end{remark}

\begin{remark} There is an alternative approach to Theorem B.
Let $c:=-\theta$. For this choice of $c$, a simple calculation shows that
\[[V,F]=H_c,\;\;\;[V,H_{c}]=-F,\;\;\;\;[F,H_{c}]=-V(\theta)F+\theta H_{c}
+(K-\mbox{\rm div}\,\e)V\]
and $F(c) + c^2 + \Kmag=K-\mbox{\rm div}\,\e$.
If we assume $\e$ to be divergence free and we apply the Pestov
identity (\ref{EQNPestovInt}) to $\varphi$, we obtain
\[ 2\int_{SM} H_c\varphi\cdot VF\varphi \, d\mu = \NormSq{F\varphi}{L^2} + \NormSq{H_c\varphi}{L^2} - \int_{SM} K (V\varphi)^2 \, d\mu. \]
But for our choice of $c$, $VF\varphi=-2H_{c}\varphi$ (cf. (\ref{ex})) and thus
\[ -5\NormSq{H_c\varphi}{L^2} = \NormSq{F\varphi}{L^2} - \int_{SM} K (V\varphi)^2 \, d\mu. \]
Hence if $K\leq 0$ we see that $F(\varphi)=H_{c}(\varphi)=0$ which implies
that $\varphi$ is constant. From this, Theorem B follows as above.

\end{remark}

\setcounter{section}{0}
\renewcommand{\thesection}{\Alph{section}}
\section{Appendix}

\subsection{A Pestov identity for thermostats on surfaces}

We modify the standard vector fields over $SM$, $\{X, H, V\}$ as defined in \cite{singer-1988-}, to \begin{displaymath} \left\{ \begin{matrix} F = X + \lambda V \\ H_c = H + c V \\ V \end{matrix} \right\} \end{displaymath} with $\lambda$ and $c$ both functions on $SM$. Notice that $F$ is the vector field induced by the thermostat flow $\ddot{\gamma} = \lambda(\gamma, \dot{\gamma}) \, i\dot{\gamma}$, and $c$ could be chosen such that $F$ and $H_c$ span a weak bundle if the bundle were always transverse to the vertical.

By a trivial calculation, the modified commutation relations are \begin{displaymath} [V, F] = H_c + (V(\lambda) - c) V, \end{displaymath} \begin{displaymath} [V, H_c] = -F +(V(c) + \lambda) V, \end{displaymath} \begin{displaymath} [F, H_c] = -\lambda F - c H_c + (F(c) + c^2 + \Kmag ) V, \end{displaymath} where $\Kmag =K-H_{c}(\lambda)+\lambda^2 $.

\begin{lemma} Rewritten in terms of the basis $\{ F, H_c, V\}$, the Pestov identity becomes \begin{equations}  2H_c u \cdot VFu &=&
 (Fu)^2 + (H_c u)^2 - (F(c) + c^2 + \Kmag)(V u)^2 \\  && \qquad + F(H_c u \cdot V u) + V(\lambda) F u \cdot V u \\ && \qquad - H_c(F u \cdot V u) - V(c) F u \cdot V u \quad + V(H_c u \cdot F u). \end{equations} \end{lemma}

\begin{proof} \begin{equations} && 2H_c u \cdot V F u - V(H_c u \cdot F u)
\\ &=& H_c u \cdot VF u - VH_c u \cdot Fu
\\&=& H_c u \cdot ( F V u + [V, F]u) - Fu \cdot (H_c V u + [V, H_c]u)
\\&=& H_c u \cdot (F V u + H_c u + (V(\lambda) - c)Vu) \\ &&\quad - Fu \cdot (H_c V u - Fu + (\lambda + V(c)) \cdot Vu)
\\&=& (H_cu)^2 + (Fu)^2 + H_cu \cdot FVu + (V(\lambda) - c) \cdot Vu \cdot H_c u \\ &&\quad - Fu \cdot H_cVu - (\lambda + V(c)) \cdot Vu \cdot Fu
\\&=& (H_cu)^2 + (Fu)^2 + F(H_cu \cdot Vu) \\ && \quad -FH_cu \cdot Vu + H_c F u \cdot Vu - H_c(Fu \cdot Vu) \\&& \quad + ((V(\lambda) - c)H_cu
- (\lambda + V(c))Fu)\cdot Vu
\\&=& (H_cu)^2 + (Fu)^2 + F(H_cu \cdot Vu) - H_c(Fu \cdot Vu) \\ && \quad + ([H_c, F]u + (V(\lambda) - c) H_c u + (\lambda + V(c))Fu) \cdot Vu
\\&=& (H_cu)^2 + (Fu)^2 + F(H_cu \cdot Vu) - H_c(Fu \cdot Vu) \\ && - \left( (F(c) + c^2 + \Kmag ) Vu - V(\lambda) H_cu - V(c)Fu \right) \cdot Vu
\\&=& (Fu)^2 + (H_c u)^2 - (F(c) + c^2 + \Kmag)(V u)^2 + \\ && \quad F(H_c u \cdot V u) + V(\lambda) F u \cdot V u \\ && \quad - H_c(F u \cdot V u) - V(c) F u \cdot V u + V(H_c u \cdot F u). \end{equations} \end{proof}

\subsection{Jacobi fields for thermostats} \label{SECTthermojacobi}

Here $\phi_t$ is the thermostat flow in the tangent bundle. We consider variations of a path $\gamma(t) = \pi \circ \phi_t (x,v)$, for some $(x,v) \in SM$. Notice $\gamma$ is defined on $M$, rather than on $SM$ as elsewhere in the paper.

We let $f(s,t) = \pi(\phi_{t}(Z(s)))$ be a variation of $\gamma$; $f$ depends on a choice of curve in $TM$ called $Z$, with $Z(0) = (x,v)$.  The image of $f$ is a small rectangle covering $\gamma$.  Further, we define nearby paths $\gamma_s = f(s,\cdot)$ and the \emph{Jacobi fields} (dependent on $Z$) as $J = \partial f/ \partial s (0, \cdot)$.

By the definition of curvature \begin{equations} \frac{D}{ds} \frac{D}{dt} \frac{\partial f}{\partial t} &=& \frac{D}{dt} \frac{D}{dt} \frac{\partial f}{\partial s} + R(\frac{\partial f}{\partial t}, \frac{\partial f}{\partial s})\frac{\partial f}{\partial t} \\ \frac{D}{ds} \Big( \frac{D \dot{\gamma_s}}{dt} \Big) \Big|_{s=0} &=& \ddot{J} + R(\dot{\gamma}, J) \dot{\gamma}. \end{equations}

Using $D/ds|_{s=0} = \nabla_J$ on vectors over $M$ and that $\lambda = \lambda( \gamma_s(t), \dot{\gamma}_s(t))$, a function on $SM$, we have 
\begin{equations} \frac{D}{ds} \Big( \frac{D \dot{\gamma_s}}{dt} \Big) \Big|_{s=0} = \frac{D}{ds} ( \lambda \, i \dot{\gamma_s})|_{s=0} &=& \frac{D}{ds} \lambda( \gamma_s(t), \dot{\gamma}_s(t))|_{s=0} \cdot i \dot{\gamma} + 0 + \lambda \, i \nabla_J (\dot{\gamma}) \\ &=& d\lambda(J, \dot{J})\, i \dot{\gamma} + \lambda \, i \frac{D}{dt} \frac{df(s,t)}{ds}\Big|_{s=0} \\&=& d\lambda(J, \dot{J}) \, i \dot{\gamma} + \lambda \, i \dot{J}. \end{equations} 
On the first line we made use of the fact that the 1-1 tensor $i$ is parallel with respect to the Riemmanian metric. Thus the generalised Jacobi equation is 
\begin{displaymath} 
\ddot{J} + R(\dot{\gamma}, J) \dot{\gamma} - d\lambda(J, \dot{J}) \, i \dot{\gamma} - \lambda i \dot{J} = 0. 
\end{displaymath}
Notice that $d\lambda (J, \dot{J}) = (\mbox{horizontal lift of }J)(\lambda) + (\mbox{vertical lift of }\dot{J})(\lambda)$.

As we are working on a surface we can study a general Jacobi field in terms of two real variables of $t$, say $x$ and $y$, with
\begin{displaymath}
J(t) = x(t) \, \dot{\gamma}(t) + y(t) \, i \dot{\gamma}(t).
\end{displaymath}
The horizontal lift of $J$ is then $xX + yH$. Further, if the Jacobi field has initial condition $\xi\in T_{(x,v)}SM$ then, $\metric{\dot{J}}{ \dot{\gamma}} = 0$ which implies $\dot{x} = \lambda y$ and that the vertical lift of $\dot{J}$ is $(\dot{y} + \lambda x)V$. Considering the generalised Jacobi equation in the $i \dot{\gamma}$ direction then,
\begin{displaymath}
\ddot{y} + \dot{\lambda} x + \lambda^2 y + K y - x X(\lambda) - y H (\lambda) - (\dot{y} + \lambda x ) V(\lambda) = 0,
\end{displaymath}
\begin{displaymath}
\Rightarrow \quad \ddot{y} - V(\lambda) \, \dot{y} + (K - H (\lambda) + \lambda^2) \, y = 0,
\end{displaymath}
since $X(\lambda) + \lambda V(\lambda) = F(\lambda) = \dot{\lambda}$, and the analysis is complete on substituting $c(t) = \dot{y}/y$.
Note that in the notation of Section \ref{SECTjacobi},
\[S(J)=\dot{J}=xS(\dot{\gamma})+yS(i\dot{\gamma})=(\lambda x+cy)i\dot{\gamma}
=(\lambda x+\dot{y})i\dot{\gamma}.\]

\bibliography{Master_no_year}
\bibliographystyle{plain}

\noindent
\textsc{DPMMS, University of Cambridge, Cambridge, CB3 0WB, UK.} \\ 
\noindent
\texttt{j.d.t.jane@dpmms.cam.ac.uk}

\noindent
\texttt{g.p.paternain@dpmms.cam.ac.uk}

\end{document}